\documentclass[10pt,reqno]{amsart}

\usepackage{amssymb,amsmath}
\usepackage{amsfonts}
\usepackage{amsthm}
\numberwithin{equation}{section}
\usepackage{enumerate}
\usepackage[english]{babel}
\usepackage[pdftex,bookmarks,colorlinks,linkcolor=red,citecolor=green]{hyperref}

\newcommand{\R}{\mathbb{R}}

\newcommand{\de}{\partial}
\newcommand{\debar}{\overline{\partial}}
\newcommand{\herm}{$(M^{2n},J,g)$ }

\newcommand{\bismut}{\nabla^B}
\newcommand{\riccib}{(\rho^B)^{(1,1)}}

\newcommand{\g}{\mathfrak{g}}

\newcommand{\xip}{\xi^\perp}

\DeclareMathAlphabet{\mathpzc}{OT1}{pzc}{m}{it}

\newtheorem{theorem}{Theorem}[section]

\newtheorem{prop}[theorem]{Proposition}
\newtheorem{lemma}[theorem]{Lemma}

\theoremstyle{definition}
\newtheorem{definition}[theorem]{Definition}
\newtheorem{example}[theorem]{Example}
\newtheorem{remark}[theorem]{Remark}

\begin{document}

\title[Static SKT metrics on Lie groups]{Static SKT metrics on Lie groups}
\author{Nicola Enrietti}
\thanks{}
\thanks{\textsc{Dipartimento di Matematica G. Peano, Universit\`a degli studi di Torino}}
\thanks{\textsc{Via Carlo Alberto 10, 10123 Turin, Italy}}
\thanks{\href{mailto:nicola.enrietti@unito.it}{nicola.enrietti@unito.it}}
\begin{abstract}\
An SKT metric is a Hermitian metric on a complex manifold whose fundamental 2-form $\omega$ satisfies $\de\debar\omega=0$. Streets and Tian introduced in \cite{sttiPlur} a Ricci-type flow that preserves the SKT condition. This flow uses the Ricci form associated to the Bismut connection, the unique Hermitian connection with totally skew-symmetric torsion, instead of the Levi-Civita connection. A SKT metric is static if the (1,1)-part of the Ricci form of the Bismut connection satisfies $\riccib=\lambda\omega$ for some real constant $\lambda$.
We study invariant static metrics on simply connected Lie groups, providing in particular a classification in dimension 4 and constructing new examples, both compact and non-compact, of static metrics in any dimension.
\end{abstract}

\maketitle

\section*{Introduction}

Let \herm be a Hermitian manifold of real dimension $2n$. We say that $g$ is \emph{Strong KT} (for short \emph{SKT}) or \emph{pluriclosed} if $\de\debar\omega =0$. This condition is strictly related to the Bismut connection \cite{bism,gaud}, which is the unique Hermitian connection whose torsion tensor is totally skew-symmetric.
SKT metrics were introduced in the context of type II string theory and $2$-dimensional supersymmetric $\sigma$-models \cite{ghr,stro}, and they have also relations with generalized K\"ahler geometry \cite{gua,fito1}. Moreover, Gauduchon \cite{gaud2} proved that for compact complex surfaces, one can find an SKT metric in the conformal class of any given Hermitian metric.

In \cite{sttiPlur,sttiPlur2} Streets and Tian introduced a parabolic flow of SKT metrics defined by
\[
\frac{\de\omega(t)}{\de t} = -\riccib,
\]
where $\rho^B$ is the Ricci form of the Bismut connection. This led to a definition of Einstein-like SKT metrics, called \emph{static}. More precisely, we say that an SKT metric $g$ on a complex manifold ($M,J$) is static if 
\[
-\riccib = \lambda \omega
\]              
for some $\lambda\in\R$. It is straightforward that every K\"ahler-Einstein metric is static, but the link between static metric and K\"ahler-Einstein metric is deeper. First of all, in \cite{sttiPlur}, it is shown that to any static metric with $\lambda\neq0$ we can associate a symplectic form that tames $J$, called Hermitian-symplectic in \cite{sttiPlur}. In \cite{lizh} it was proved that if $(M,J)$ is a compact complex surface, then the existence of a symplectic form that tames $J$ implies the existence of a K\"ahler metric on $(M,J)$. Moreover, a nilmanifold, i.e. the compact quotient of a nilpotent simply connected Lie group by a discrete subgroup, endowed with an invariant complex structure $J$ cannot admit any K\"ahler metric (\cite{bego,has}), and in \cite{efv} it is proved that it cannot admit any symplectic form that tames $J$, either. Indeed, it is still an open problem to find an example of a complex manifold admitting a symplectic form that tames $J$, but no K\"ahler structures. 

Secondly, all the examples of static metrics contained in \cite{sttiPlur} are K\"ahler-Einstein except the Hopf manifold, that admit a static metric with $\lambda=0$. In fact, we prove that on a compact K\"ahler manifold any static metric induces a K\"ahler-Einstein metric, so examples of non K\"ahler-Einstein static metrics on compact manifolds have to be found on non-K\"ahler manifolds.

\medskip

Our investigation concern in particular Lie groups and compact quotients of Lie groups by discrete subgoups. The study of SKT metrics on such manifolds was developed in \cite{fps,masw,efv,swan,uga}. 

First, we focus on nilmanifolds, and we prove that no invariant static metrics can be found on nilmanifods toghether with invariant complex structures (with the exception of tori).

Then we classify all the invariant static metrics on simply connected Lie groups of dimension 4, obtaining that the unique 4-dimensional Lie algebra that admit a non K\"ahler-Einstein invariant static metric is $\mathfrak{su}(2)\times\R$. This is not surprising, because it is the Lie algebra associated to the group $S^3\times S^1$, that is diffeomorphic to the Hopf manifold. So the Hopf manifold admits static metrics (induced by the invariant ones), that is the same result of \cite{sttiPlur}.

In the last section, we focus on static metrics with $\lambda=0$. First, we prove that if $(\g,J,g,D)$ is a Lie algebra together with a static metric $g$ with $\lambda=0$ and a flat Hermitian connection $D$ then on the tangent Lie algebra $T_D\,\g=\g\ltimes_D \R^{2n}$ we can produce a static metric with $\lambda=0$. Then we note that for every compact even-dimensional semisimple Lie group, the bi-invariant metric is a static metric with $\lambda=0$ with respect to any compatible complex structure, so we can apply the tangent Lie algebra construction to such groups, obtaining the first examples of compact and non-compact static metrics in (real) dimension greater than 4.
\medskip

\noindent\textbf{Acknowledgements.} The author is grateful to Anna Fino who proposed him the subject of this paper. He would also like to thank Sergio Console and Luigi Vezzoni for useful comments.

\bigskip

\section{Preliminaries}

We start by recalling some definitions and fixing some notation. Let $(M^{2n},J)$ be a complex manifold of real dimension $2n$ and $g$ be a Hermitian metric on $(M,J)$ with fundamental 2-form defined by $\omega(\cdot,\cdot)=g(\cdot,J\cdot)$. By \cite{gaud} there exist a unique connection $\nabla^B$ on $M$, called the \emph{Bismut connection}, such that $\bismut J=\bismut g=0$ and whose torsion $3$-form
\[
c(X,Y,Z)=g(X,T^B(Y,Z))
\]
is totally skew-symmetric. It is well known that $c=-Jd\omega$.

\begin{definition}
A Hermitian metric $g$ on a complex manifold $(M,J)$ is \emph{strong K\"ahler with torsion} or \emph{SKT} if the torsion 3-form $c$ of the Bismut connection $\bismut$ is closed, i.e. $dc=0$. This condition is equivalent to $\de\debar\omega=0$.
\end{definition}
Since $\bismut$ is a Hermitian connection, we can define the \emph{Ricci form} of $\bismut$ as
\[
\rho^B(X,Y) = \frac{1}{2}\sum_{k=1}^{2n} g(R^B(X,Y) \mathbf{e}_k,J \mathbf{e}_k),
\]
where $\{  \mathbf{e}_i \}$ is a local orthonormal frame of the tangent bundle $TM$ and $R^B$ is the curvature tensor of $\bismut$ defined by
\[
R^B(X,Y)Z=\nabla^B_{[X,Y]}Z-[\nabla^B_X,\nabla^B_Y]Z. 
\]
In the same way, we can define the Ricci form of the Chern connection, i.e. the unique Hermitian connection such that the $(1,1)$-part of the torsion tensor vanishes. The Ricci form of the Chern connection is related to the one of the Bismut connection by the formula (\cite{aliv,figr})
\begin{equation}\label{bischern}
\rho^B=\rho^C+dd^*\omega.
\end{equation}
In \cite{sttiPlur,sttiPlur2} Streets and Tian studied the evolution equation
\begin{equation}\label{flow}
\left\{
\begin{aligned}
&\frac{\de\omega(t)}{\de t} = -\riccib \\
&\omega(0)=\omega_0
\end{aligned}
\right.
\end{equation}
where $\riccib$ is the projection of $\rho^B$ on the bundle of (1,1)-forms. This flow preserves the SKT condition and is elliptic on the set of SKT metrics, so short-time existence of solutions is guaranteed. Moreover, if the initial condition $\omega_0$ is K\"ahler then \eqref{flow} coincides with the K\"ahler-Ricci flow \cite{cao}.

Equation \eqref{flow} allows to define an Einstein-like condition for SKT metrics.
\begin{definition}[\cite{sttiPlur}]
A Hermitian metric $g$ on a complex manifold $(M,J)$ is called \emph{static} if it is SKT and the Ricci tensor of the Bismut connection satisfies 
\begin{equation}\label{static}
-\riccib = \lambda \omega
\end{equation}
for some real constant $\lambda$.
\end{definition}

As pointed out in \cite{sttiPlur} static metrics with $\lambda\neq0$ carries additional structures. Indeed, if $\lambda\neq0$, then 
\[
-\frac{1}{\lambda}\rho^B(JX,X)>0\qquad \text{and}\qquad d\rho^B=0,
\]
so $-\frac{1}{\lambda}\rho^B$ is a symplectic form and tames the complex structures $J$.

\begin{theorem}
Let $(M,J)$ be a compact complex manifold, and suppose that it admits a K\"ahler metric. Then the existence of a static metric is equivalent to the existence of a K\"ahler-Einstein metric. In particular,
\begin{itemize}
\item If $g$ is a static metric with $\lambda\neq0$, then $g$ is itself a K\"ahler-Einstein metric;
\item If $g$ is a static metric with $\lambda=0$, then $(M,J)$ is Calabi-Yau manifold.
\end{itemize}
\end{theorem}

\begin{lemma}\label{lemma}
Let $(M,J)$ be a complex manifold and $g$ a static metric with $\lambda\neq 0$ such that $\rho^B\in\Omega^{1,1}(M)$. Then $g$ is K\"ahler-Einstein.
\end{lemma}
\begin{proof}
If $g$ is a static metric with $\lambda\neq0$ and $\rho^B\in\Omega^{1,1}(M)$, then $\omega=-\frac{1}{\lambda}\rho^B$. But $\rho^B=dd^*\omega+\rho^C$ is closed, so $d\omega=0$.
\end{proof}

\begin{proof}[Proof of Theorem 0.1]
Clearly, if $g$ is a K\"ahler-Einstein metric, it satisfies the static condition.

Now suppose that $g$ is a static metric. Since $(M,J)$ admits a K\"ahler metric, the $\de\debar$-lemma holds. Then there is a function $f$ on $M$ such that $dd^*\omega=\de\debar f$, so $dd^*\omega$ is of type $(1,1)$. Therefore $\rho^B\in\Omega^{1,1}(M)$. 

If $\lambda\neq0$, appllying Lemma \ref{lemma} we have that $g$ is K\"ahler-Einstein.

If $\lambda=0$, then $\rho^B=0$. But in general $[\rho^B]=[\rho^C]=c_1\in H^2(M,\R)$, so $c_1=0$, and since $(M,J)$ admits a K\"ahler metric it is Calabi-Yau.
\end{proof}

\begin{remark}
Lemma \ref{lemma} holds for any complex manifold, either compact or non-compact.
\end{remark}

\bigskip

\section{Nilmanifolds}\ 

We recall that a \emph{nilmanifold} is the compact quotient of a simply connected nilpotent Lie group $G$ by a discrete subgroup $\Gamma$. By invariant Riemannian metric (respectively complex structure) on $G/\Gamma$ we mean the one induced by an inner product (respectively complex structure) on the Lie algebra $\g$ of $G$. It is well known that a nilmanifold cannot admit any K\"ahler metric unless it is a torus (see for example \cite{bego,has}), and results about classification of SKT metrics on nilmanifolds have been found in \cite{fps,efv}. Moreover, in \cite{efv} it is proved that a nilmanifold (not a torus) together with an invariant complex structure $J$ cannot admit any symplectic form taming $J$, so in particular we cannot find any static metric with $\lambda\neq0$. We wonder what happens for $\lambda=0$.

Since we are considering invariant metrics, we can work on nilpotent Lie algebras. We recall that a Lie algebra $\g$ is \emph{nilpotent} if the descending central series $\{ \g^k \}_{k\geqslant0}$ defined by
\[
\g^0=\g,\quad \g^1=[\g,\g] \quad \dots  \quad \g^k=[\g^{k-1},\g]
\]
vanishes for some $k>0$. By \cite{efv} any SKT-nilpotent Lie algebra $\g$ is 2-step (i.e. $\g^2=\{0\}$) and its center is $J$-invariant; therefore we can split $\g$ in $\xi\oplus\xi^\perp$, where $\xi$ is the center, $\xip$ the orthogonal complement to the center with respect to the SKT metric and $[\xip,\xip]\subset\xi$, so for every $X\in\g$ we have a unique decomposition $X=X^\xi+X^\perp$, where $X^\xi\in\xi$ and $X^\perp\in\xip$.

In the following lemmas we make some calculations about the Bismut connection and the SKT condition:

\begin{lemma}\label{lemma1}
 Let $\g$ be a nilpotent Lie algebra together with a complex structure $J$ and a $J$-Hermitian SKT metric $g$, and $\bismut$ its Bismut connection. Then for any $X,Y\in\g$
 \begin{enumerate}[i)]
  \item\label{lemma1.1} $\bismut_{X^\xi}Y^\xi=0$; \vspace{1,5mm}   
  \item\label{lemma1.2} $\bismut_{X^\xi}Y^\perp\in\xip$ and $g(\bismut_{X^\xi}Y^\perp,Z)=-\frac{1}{2}g([Y^\perp,Z]+[JY^\perp,JZ],X^\xi);$\vspace{1,5mm}   
  \item\label{lemma1.3} $\bismut_{X^\perp}Y^\xi\in\xip$ and $g(\bismut_{X^\perp}Y^\xi,Z)=-\frac{1}{2}g([X^\perp,Z]-[JX^\perp,JZ],Y^\xi).$ Moreover,
   \begin{equation}\label{eqlemma1}
    J\bismut_{JX^\perp}Y^\xi = \bismut_{X^\perp}Y^\xi
   \end{equation}
  \item\label{lemma1.4} $\bismut_{X^\perp}Y^\perp=\frac{1}{2}([X^\perp,Y^\perp]-[JX^\perp,JY^\perp]) \in\xi$.
 \end{enumerate}

\end{lemma}

\begin{proof}
In view of \cite{dofi} we can write the Bismut connection in terms of Lie brackets as
\begin{equation}\label{eqbismut}
 \tilde g(\bismut_XY,Z) = \frac{1}{2} \Big\{ \tilde{g}([X,Y]-[JX,JY],Z) - \tilde{g}([Y,Z]+[JY,JZ],X) - \tilde{g}([X,Z]-[JX,JZ],Y) \Big\}.
\end{equation}
Relations \emph{(\ref{lemma1.1}),(\ref{lemma1.2}),(\ref{lemma1.4})} and the first part of \emph{(\ref{lemma1.3})} comes directly by using the definition of $\xi$. Equation \eqref{eqlemma1} can be obtained using the first part of \emph{(\ref{lemma1.3})} and the integrability of $J$.

\end{proof}

\begin{lemma}\label{lemma2}
 Let $\g$ be a nilpotent Lie algebra together with a complex structure $J$ and a $J$-Hermitian SKT metric $g$. Then  
 \[
  g([X,JX],[Y,JY])= \frac{1}{2}\big( \Vert [X,Y] \Vert^2 + \Vert [X,JY] \Vert^2 + \Vert [JX,Y] \Vert^2 + \Vert [JX,JY] \Vert^2 \big) 
 \]
for every $X,Y\in\g$.
\end{lemma}

\begin{proof}
 If $X$ or $Y$ belongs to the center, then the lemma is obviously true; so we consider the case $X,Y\in\xip$. We can write $c$ in terms of Lie brackets as
\begin{equation}\label{eqtorsion}
 c(X,Y,Z)=-g([JX,JY],Z)-g([JY,JZ],X)-g([JZ,JX],Y), 
\end{equation}
then
\[
\begin{split}
 0 = dc(X,Y,JX,JY) =&\, -c([X,Y], JX, JY) + c([X, JX], Y, JY) - c([X, JY], Y, JX)- \\
                    &\,- c([Y, JX], X, JY) + c([Y, JY], X, JX) - c([JX, JY], X, Y)\\
                   =&\, +g ([X,Y], [X,Y]) -  g([Y, JY], [X, JX])+g([X, JY], [X, JY])  \\
                    &\, + g ([Y, JX], [Y, JX])- g([Y, JY], [X, JX]) + g([JX, JY], [JX, JY]) \\
                   =&\, -2g([X, JX], [Y, JY]) + \Vert [X,Y] \Vert^2 + \Vert [X,JY] \Vert^2 + \Vert [JX,Y] \Vert^2+ \\
                    &\,  + \Vert [JX,JY] \Vert^2 \\
\end{split}
\]
as required.

\end{proof}

Now we are ready to prove the following

\begin{theorem}
Let $G/\Gamma$ a nilmanifold (not a torus) together with an invariant complex structure $J$. Then it does not admit any $J$-Hermitian invariant static metric with $\lambda=0$.
\end{theorem}

\begin{proof}
Let $\g$ the Lie algebra of $G$, $\tilde J$ the induced complex structure and $g$ a $\tilde J$-Hermitian SKT metric; we have $\g=\xi\oplus\xip$. Choose $\{  \mathbf{e}_1,\dots, \mathbf{e}_{2m} \}$ and $\{  \mathbf{f}_1,\dots, \mathbf{f}_{2k} \}$ to be orthonormal basis respectively of $\xip$ and $\xi$ with $2m+2k=2n=\dim\g$; then $\{  \mathbf{e}_1,\dots, \mathbf{e}_{2m}, \mathbf{f}_1,\dots, \mathbf{f}_{2k} \}$ is an orthonormal basis of $\g$. Note that $\riccib(X,\tilde JX)=\rho^B(X,\tilde JX)$, so in order to prove that $\riccib\neq 0$ we will show that $\rho^B(X,\tilde JX)$ is not zero for some $X\in\g$.

Suppose $X\in\xip$; by definition,
\[
\rho^B(X,\tilde JX) = \frac{1}{2}\Big( \sum_{i=1}^{2m} g(R^B(X,\tilde JX) \mathbf{e}_i,\tilde J \mathbf{e}_i) + \sum_{j=1}^{2k} g(R^B(X,\tilde JX) \mathbf{f}_j,\tilde J \mathbf{f}_j)\, \Big);
\]
we consider the two summations separately.
\begin{itemize}

\item By definition of $R^B$, we obtain
 \[
  g(R^B(X,\tilde JX) \mathbf{e}_i,\tilde J \mathbf{e}_i) = g(\bismut_X\bismut_{\tilde JX} \mathbf{e}_i,\tilde J \mathbf{e}_i)-g(\bismut_{\tilde JX}\bismut_X  \mathbf{e}_i,\tilde J \mathbf{e}_i) -g(\bismut_{[X,\tilde JX]} \mathbf{e}_i,\tilde J \mathbf{e}_i).
 \]
Applying Lemma \ref{lemma1} and using the integrability of $\tilde J$ we have
\[
 g(\bismut_X\bismut_{\tilde JX} \mathbf{e}_i,\tilde J \mathbf{e}_i) = -g(\bismut_{\tilde JX}\bismut_X  \mathbf{e}_i,\tilde J \mathbf{e}_i) = -\frac{1}{4}\Vert [X, \mathbf{e}_i]-[\tilde JX,\tilde J \mathbf{e}_i] \Vert^2
\]
and
\[
 g(\bismut_{[X,\tilde JX]} \mathbf{e}_i,\tilde J \mathbf{e}_i) = -g([X,\tilde JX],[ \mathbf{e}_i,\tilde J \mathbf{e}_i]).
\]
Hence
\begin{equation}\label{eqnilp1}
 g(R^B(X,\tilde JX) \mathbf{e}_i,\tilde J \mathbf{e}_i) = -\frac{1}{2}\Vert [X, \mathbf{e}_i]-[\tilde JX,\tilde J \mathbf{e}_i] \Vert^2 + g([X,\tilde JX],[ \mathbf{e}_i,\tilde J \mathbf{e}_i]).
\end{equation}

\item Again by definition of $R^B$ and applying Lemma \ref{lemma1} and equation \eqref{eqlemma1}, we obtain
 \begin{align*}
   g(R^B(X,\tilde JX) \mathbf{f}_j,\tilde J \mathbf{f}_j) =&\, g(\bismut_X\bismut_{\tilde JX} \mathbf{f}_j,\tilde J \mathbf{f}_j)-g(\bismut_{\tilde JX}\bismut_X \mathbf{f}_j,\tilde J \mathbf{f}_j) \\
                        =&\, \frac{1}{2} g([X,\bismut_{\tilde JX} \mathbf{f}_j-\tilde J\bismut_X \mathbf{f}_j] - [\tilde JX,\tilde J\bismut_{\tilde JX} \mathbf{f}_j+\bismut_X \mathbf{f}_j],\tilde J \mathbf{f}_j) \\
                        =&\, g( [X,\bismut_{\tilde JX} \mathbf{f}_j] - [\tilde JX,\tilde J\bismut_{\tilde JX} \mathbf{f}_j] ,\tilde J \mathbf{f}_j).
 \end{align*}
By decomposing $\bismut_{\tilde JX} \mathbf{f}_j$ in components with respect to the basis $\{ \mathbf{e}_i \}$ of $\xip$ we compute that
\begin{equation}\label{eqnilp2}
 \sum_{j=1}^{2k} g(R^B(X,\tilde JX) \mathbf{f}_j,\tilde J \mathbf{f}_j) = \frac{1}{2}\sum_{i=1}^{2m} \Vert [X, \mathbf{e}_i]-[\tilde JX,\tilde J \mathbf{e}_i] \Vert^2.
\end{equation}
 
\end{itemize}

Combining equations \eqref{eqnilp1} and \eqref{eqnilp2} we obtain
\begin{align*}
 \rho^B(X,\tilde JX) = &\,\frac{1}{2}\sum_{i=1}^{2m} g([X,\tilde JX],[ \mathbf{e}_i,\tilde J \mathbf{e}_i]) \\
\intertext{and using Lemma \ref{lemma2}}
              =&\, \frac{1}{4}\sum_{i=1}^{2m} \big(\, \Vert [X, \mathbf{e}_i] \Vert^2 + \Vert [X,\tilde J \mathbf{e}_i] \Vert^2 + \Vert [\tilde JX, \mathbf{e}_i] \Vert^2 + \Vert [\tilde JX,\tilde J \mathbf{e}_i] \Vert^2\, \big) > 0
\end{align*}
since $X\in\xip$; this concludes the proof.

\end{proof}

The results of this section can be summarized as follows: \emph{let $G/\Gamma$ be a nilmanifold (not a torus) endowed with an invariant complex structure $J$ and with a $J$-invariant SKT metric $g$; then, if $g$ is a static metric, it must be non-invariant and $\lambda$ must be zero.}
Whether such metrics exists is still not known, but a possible approach to the problem could be the following: let $g$ be a non-invariant Hermitian metric on $(G/\Gamma,J)$, with $J$ invariant. By \cite{miln} $G/\Gamma$ has a bi-invariant volume form $d\mu$, and applying the symmetrization process of \cite{belg} we can construct a new invariant $J$-Hermitian metric $\tilde g$ by posing
\[
\tilde g (X,Y) = \int_{m\in M} g_m(X_m,Y_m) d\mu
\]
for any left-invariant vector fields $X,Y$. Moreover, in \cite{uga} it was proved that if the metric $g$ is SKT, then $\tilde g$ is still SKT.
Thus, if a nilmanifold admits a non-invariant static metric with $\lambda=0$, then it induces an invariant SKT metric $\tilde g$. In general, however, it is not true that the Ricci form $\tilde \rho^B$ of the metric $\tilde g$ is obtained by the symmetrization of the Ricci tensor $\rho^B$ of $g$, so it is an open problem to check if the induced invariant metric $\tilde g$ is still static.

\bigskip

\section{Static metrics in dimension 4}
The study of static metrics in dimension 4 turns out to be strictly related to K\"ahler-Einstein metrics. Indeed, combining Lemma 4.4 of \cite{aliv} and Theorem 2 of \cite{gaiv} we have that if $(M,J)$ is a compact complex surface and $g$ is static metric on $(M,J)$, then either $(M,J,g)$ is K\"ahler-Einstein or $(M,J)$ is the Hopf surface. Moreover, in \cite{gaiv,sttiPlur} it was proved that the Hopf surface admits a static metric with $\lambda=0$, and the Hopf surface cannot admit any symplectic form that tames the complex structure \cite[Proposition 2.24]{gua}. So
\begin{theorem}
Let $(M,J)$ be a compact complex surface and $g$ a static metric. Then one of the following cases occurs:
\begin{itemize}
\item $(M,J,g)$ is K\"ahler-Einstein 
\item $(M,J)$ is the Hopf surface and $\lambda=0$.
\end{itemize}
\end{theorem}

In this section we prove a similar result for left-invariant static metrics on 4-dimensional simply connected Lie groups (not necessarily compact) endowed with a left-invariant complex structure.

\begin{theorem}\label{dim4}
Let $(G,J,g)$ be a simply connected Lie group together with an invariant complex structure $J$ and an invariant static metric $g$. Then one of the following cases occurs:
\begin{itemize}
\item $(G,J,g)$ is K\"ahler-Einstein 
\item The Lie algebra $\g$ of $G$ is isomorphic to $\mathfrak{su}(2)\times\R$ and $\lambda=0$.
\end{itemize}
\end{theorem}

Note that the Lie algebra $\mathfrak{su}(2)\times\R$ plays the same role as the Hopf manifold for compact complex surfaces. This happens because $\mathfrak{su}(2)\times\R$ is the Lie algebra of the Lie group $S^3\times S^1$, that is diffeomorphic to the Hopf manifold.

\begin{proof}[Proof of Theorem \ref{dim4}]
Since we are interested in invariant structures, it is sufficient to study the induced structures on the corresponding Lie algebra.
Let $\g$ be a Lie algebra: the \emph{derived series} of $\g$ is defined by $\mathcal D^1\g=[\g,\g],\ \mathcal D^k\g=[\mathcal D^{k-1}\g,\mathcal D^{k-1}\g]$, and we say that $\g$ is \emph{solvable} if there exists an integer $s$ such that $\mathcal D^s\g=0$. According to \cite{bb}, a Lie algebra of dimension 4 is either solvable, isomorphic to $\mathfrak{su}(2)\times\R$ or isomorphic to $\mathfrak{sl} (2,\R) \times\R$. We started by considering solvable Lie algebras.
\vspace{3mm}

A classification of 4-dimensional solvable Lie algebras admitting a left-invariant complex structure can be found in \cite{snow}, and recently Madsen and Swann in \cite{masw} gave a classification of SKT structures on solvable Lie algebras of dimension four.
With the help of a \textsc{Maple} software, we use this classification to compute directly the Ricci tensor of the Bismut connection.
According to \cite{masw}, we can suppose that $J$ is defined by $Je^1=e^2,\ Je^3=e^4$, where $\{ e^1,e^2,e^3,e^4 \}$ is a basis of $\g^*$. Moreover, $\g$ belongs to one of the following cases:
\vspace{2mm}

 \textbf{Complex case}: $\g$ has structure equations
\[
\begin{cases} 
d \mathbf{e}^1 = 0 \\
d \mathbf{e}^2 = a_1 \mathbf{e}^{12} \\
d \mathbf{e}^3 = b_1 \mathbf{e}^{12}+b_2 \mathbf{e}^{13}+b_3 \mathbf{e}^{14}-c_1 \mathbf{e}^{23}+c_2 \mathbf{e}^{24} \\
d \mathbf{e}^4 = d_1 \mathbf{e}^{12}+d_2 \mathbf{e}^{13}+d_3 \mathbf{e}^{14}-f_1 \mathbf{e}^{23}+f_2 \mathbf{e}^{24}+h_1 \mathbf{e}^{34},
\end{cases}
\]
where by $e^{ij}$ we denote the wedge product $e^i\wedge e^j$, and the real coefficients $a_1,b_i,c_i,d_i,f_i,h_1$ satisfy
\begin{equation}\label{eqcomplex}
\begin{array}{lll}
f_{{1}}=c_{{2}}+d_{{3}}-b_{{2}} && f_{{2}}=-c_{{1}}+d_{{2}}+b_{{3}} \\
a_{{1}}c_{{1}}-b_{{3}}f_{{1}}-c_{{2}}d_{{2}}=0 && c_{{2}}a_{{1}}-c_{{2}}b_{{2}}+c_{{2}}d_{{3}}-b_{{3}}c_{{1}}-b_{{3}}f_{{2}}=0 \\
h_{{1}} \left( {b_{{2}}}^{2}+{b_{{3}}}^{2}+{c_{{1}}}^{2}+{c_{{2}}}^{2} \right)=0 && f_{{1}}a_{{1}}+f_{{1}}b_{{2}}-f_{{1}}d_{{3}}-d_{{2}}c_{{1}}-f_{{2}}d_{{2}}+h_{{1}}d_{{1}}=0 \\
a_{{1}}f_{{2}}+b_{{1}}h_{{1}}-b_{{3}}f_{{1}}-c_{{2}}d_{{2}}=0 &&
 \left( a_{{1}}+b_{{2}}+d_{{3}} \right)  \left( b_{{2}}+d_{{3}}
 \right) + \left( c_{{1}}-f_{{2}} \right) ^{2}-h_{{1}}d_{{1}}=0.
\end{array}
\end{equation}
In the sequel, to shorten the notation, we will denote the structure equations of $\g$ as
\[
\g = (0,a_1 \mathbf{e}^{12},b_1 \mathbf{e}^{12}+b_2 \mathbf{e}^{13}+b_3 \mathbf{e}^{14}-c_1 \mathbf{e}^{23}+c_2 \mathbf{e}^{24},d_1 \mathbf{e}^{12}+d_2 \mathbf{e}^{13}+d_3 \mathbf{e}^{14}-f_1 \mathbf{e}^{23}+f_2 \mathbf{e}^{24}+h_1 \mathbf{e}^{34}).
\]
The fundamental 2-form of the SKT metric is $\omega =  \mathbf{e}^{12}+ \mathbf{e}^{34}$, and we obtain
\[
\begin{split}
\rho^B(X,Y) =&\, ({a_{{1}}}^{2}+{b_{{1}}}^{2}+{d_{{1}}}^{2}+a_{{1}}c_{{2}}-a_{{1}}b_{{2}}+h_{{1}}d_{{1}})\cdot\mathbf{e}^{12}  +  (b_{{1}}b_{{2}}+d_{{1}}d_{{2}}+h_{{1}}d_{{2}}) \cdot\mathbf{e}^{13}  + \\ &\,(b_{{1}}b_{{3}}+d_{{1}}d_{{3}}+h_{{1}}d_{{3}})\cdot\mathbf{e}^{14}  +  (d_{{1}}b_{{2}}-b_{{1}}c_{{1}}-d_{{1}}c_{{2}}-d_{{1}}d_{{3}}-h_{{1}}c_{{2}}-h_{{1}}d_{{3}}+h_{{1}}b_{{2}})\cdot\mathbf{e}^{23}+ \\
&\,(d_{{1}}b_{{3}}+b_{{1}}c_{{2}}-d_{{1}}c_{{1}}+d_{{1}}d_{{2}}-h_{{1}}c_{{1}}+h_{{1}}d_{{2}}+h_{{1}}b_{{3}})\cdot \mathbf{e}^{24}  +  (h_{{1}}d_{{1}}+{h_{{1}}}^{2})\cdot  \mathbf{e}^{34}. 
\end{split}
\]
If we impose that $-\riccib=\lambda\omega$ with $\lambda\neq0$, we find that $h_1\neq0$, so by \eqref{eqcomplex} $b_2=b_3=c_1=c_2=0$, and the $(2,0)+(0,2)$-part of $\rho^B$ vanishes. Then by Lemma \ref{lemma} any static metric is K\"ahler-Einstein. \\
On the other hand, imposing $\riccib=0$ we find that $d\omega=0$, so $g$ must be K\"ahler-Einstein.
\vspace{2mm}

 \textbf{Real case I}: $\g$ has structure equations
\[
(0,a_1 \mathbf{e}^{12}+a_3( \mathbf{e}^{14}- \mathbf{e}^{23})+b_2 \mathbf{e}^{34},0,d_1 \mathbf{e}^{12}+d_3( \mathbf{e}^{14}- \mathbf{e}^{23})+h_1 \mathbf{e}^{34})
\]
where $d \mathbf{e}^2$ and $d \mathbf{e}^4$ are linearly indipendent and the real coefficients satisfy
\begin{equation}\label{eqrealI}
\begin{array}{lll}
b_{{2}}a_{{1}}-b_{{2}}d_{{3}}+f_{{2}}a_{{3}}-{a_{{3}}}^{2}=0 &&
d_{{1}}f_{{2}}-d_{{1}}a_{{3}}+d_{{3}}a_{{1}}-{d_{{3}}}^{2}=0 \\
d_{{3}}a_{{3}}-b_{{2}}d_{{1}}=0 && 
(d_1-a_3)(f_2+a_3)-(d_3+a_1)(d_3-b_2)=0.
\end{array}
\end{equation}
The fundamental 2-form of the SKT metric is $\omega =  \mathbf{e}^{12}+ \mathbf{e}^{34}+t \mathbf{e}^{14}+t \mathbf{e}^{23}$ with $t\in (-1,1)$. Computing $\rho^B$ we find
\[
\begin{split}
\rho^B(X,Y) =&\, -{\frac {b_{{2}}a_{{1}}+f_{{2}}d_{{1}}+{a_{{1}}}^{2}+{d_{{1}}}^{2}}{{t}^{2}-1}}\cdot\mathbf{e}^{12} -{\frac {{b_{{2}}}^{2}+{f_{{2}}}^{2}+b_{{2}}a_{{1}}+f_{{2}}d_{{1}}}{{t}^{2}-1}}\cdot  \mathbf{e}^{34}    \\
&\,-{\frac {b_{{2}}a_{{3}}+d_{{3}}f_{{2}}+a_{{3}}a_{{1}}+d_{{3}}d_{{1}}}{{t}^{2}-1}}\cdot\mathbf{e}^{14}  +  {\frac {b_{{2}}a_{{3}}+d_{{3}}f_{{2}}+a_{{3}}a_{{1}}+d_{{3}}d_{{1}}}{{t}^{2}-1}}\cdot\mathbf{e}^{23}. 
\end{split}
\]
Clearly $(\rho^B)^{(2,0)+(0,2)}=0$, so applying Lemma \ref{lemma} any static metric with $\lambda\neq0$ is K\"ahler-Einstein. Moreover, if we impose $\riccib=0$ we obtain that $\g$ must be abelian.
\vspace{2mm}

 \textbf{Real case II}: $\g$ has structure equations
\[\left\{
\begin{aligned}
&d \mathbf{e}^1 = 0 \\
&d \mathbf{e}^2 = -kq^2\, \mathbf{e}^{12}-kqr( \mathbf{e}^{14}- \mathbf{e}^{23})-kr^2\, \mathbf{e}^{34} \\
&d \mathbf{e}^3 = \frac{c_3q}{r}\, \mathbf{e}^{12}+c_3\, \mathbf{e}^{14} \\
&d \mathbf{e}^4 = \frac{kq^3}{r}\, \mathbf{e}^{12}-c_3\, \mathbf{e}^{13}+kq^2( \mathbf{e}^{14}- \mathbf{e}^{23})+kqr\, \mathbf{e}^{34},
\end{aligned}
\right.\]
with $q,r,k\in\R$ such that $q^2+r^2=1,\ r>0$ and $k\neq0$. The fundamental 2-form of the SKT metric is $\omega =  \mathbf{e}^{12}+ \mathbf{e}^{34}+t \mathbf{e}^{14}+t \mathbf{e}^{23}$, with $t\in (-1,1)$, and it is never K\"ahler. Computing $\rho^B$ we find
\[
\begin{split}
\rho^B(X,Y) =&\, 
-{\frac { \left( c_{{3}}k{q}^{2}t+{k}^{2}q\sqrt {1-{q}^{2}}+\sqrt {1-{q}^{2}}q{c_{{3}}}^{2}-ktc_{{3}} \right) q}{ \left( 1-{q}^{2} \right) ^{3/2} \left( -1+{t}^{2} \right) }} \cdot\mathbf{e}^{12}
+ {\frac {c_{{3}} \left( kq+\sqrt {1-{q}^{2}}c_{{3}}t \right) }{ \left(-1+{t}^{2} \right) \sqrt {1-{q}^{2}}}} \cdot\mathbf{e}^{13}\\ 
&\, + {\frac {t\sqrt {1-{q}^{2}}kc_{{3}}-{k}^{2}q-q{c_{{3}}}^{2}}{ \left( -1+{t}^{2} \right) \sqrt {1-{q}^{2}}}} \cdot\mathbf{e}^{14}
- {\frac {{k}^{2}q \left( -2\,{q}^{2}+2\,{t}^{2}{q}^{2}+1 \right) }{\sqrt {1-{q}^{2}} \left( -1+{t}^{2} \right) }}\cdot\mathbf{e}^{23}
- {\frac {{k}^{2}}{-1+{t}^{2}}} \cdot\mathbf{e}^{34}.
\end{split}
\]
By \cite{ef}, this Lie algebra does not admit any symplectic form that tames $J$, so we cannot have a static metric on $\g$ with $\lambda\neq 0$. On the other hand, if $\riccib=0$ we must have $k=0$, that is a contradiction. Therefore $\g$ does not admit any static metric.
\vspace{2mm}

 \textbf{Real case III}: $\g$ has structure equations
\[
\left\{
\begin{aligned}
&d \mathbf{e}^1 = 0 \\
&d \mathbf{e}^2 = -k(1+q^2)\, \mathbf{e}^{12}-kqr( \mathbf{e}^{14}- \mathbf{e}^{23})-kr^2\, \mathbf{e}^{34} \\
&d \mathbf{e}^3 = \frac{c_3q}{r}\, \mathbf{e}^{12}-\frac{k}{2}\, \mathbf{e}^{13}+c_3\, \mathbf{e}^{14} \\
&d \mathbf{e}^4 = \frac{q}{r}(kq^2+\frac{k}{2})\, \mathbf{e}^{12}-c_3\, \mathbf{e}^{13}+(kq^2-\frac{k}{2}) \mathbf{e}^{14}-kq^2 \mathbf{e}^{23}+kqr\, \mathbf{e}^{34},
\end{aligned}
\right.
\]
with $q,r,k\in\R$ such that $q^2+r^2=1,\ r>0$ and $k\neq0$; if $c_3=0$ we have $\g\cong \mathfrak{d}_{4,\frac{1}{2}} $, otherwise $\g\cong \mathfrak{d}'_{ 4,\vert\frac{ k}{2c_3}\vert }$. The fundamental 2-form of the SKT metric is $\omega =  \mathbf{e}^{12}+ \mathbf{e}^{34}+t \mathbf{e}^{14}+t \mathbf{e}^{23}$ with $t\in (-1,1)$, and is K\"ahler if and only if $q=0$. Computing $\rho^B$ we find
\[
\begin{split}
\rho^B(X,Y) =&\, 
-\frac{1}{4}\,{\frac {-8\,{k}^{2}{q}^{6}+4\,{k}^{2}{q}^{4}+8\,c_{{3}}\sqrt {1-{q}^{2}}k{q}^{3}t+{k}^{2}{q}^{2}-4\,{c_{{3}}}^{2}{q}^{2}+12\,c_{{3}}kq\sqrt {1-{q}^{2}}t-6\,{k}^{2}}{ \left( -1+{q}^{2} \right)  \left( {t}^{2}-1 \right) }} \cdot\mathbf{e}^{12}\\
&\,-\frac{1}{4}\,{\frac {8\,c_{{3}}k{q}^{3}-4\,kqc_{{3}}-4\,\sqrt {1-{q}^{2}}{c_{{3}}}^{2}t+3\,t{k}^{2}\sqrt {1-{q}^{2}}}{ \left( {t}^{2}-1 \right) \sqrt {1-{q}^{2}}}} \cdot\mathbf{e}^{13}\\ 
&\, -\frac{1}{4\left( {t}^{2}-1 \right) \sqrt {1-{q}^{2}}}\,\Big(8\,{t}^{3}k{q}^{2}c_{{3}}\sqrt {1-{q}^{2}}-8\,k{q}^{2}\sqrt {1-{q}^{2}}c_{{3}}t+16\,k\sqrt {1-{q}^{2}} {t}^{3}c_{{3}}\\
&-8\,c_{{3}}kt\sqrt {1-{q}^{2}}+16\,{t}^{2}{k}^{2}{q}^{5}-8\,{k}^{2}{q}^{5}-10 \,{t}^{2}{k}^{2}q+5\,{k}^{2}q -8\,{t}^{2}q{c_{{3}}}^{2}+4\, q{c_{{3}}}^{2}\Big) \cdot\mathbf{e}^{14}\\
&\,+\frac{1}{2}\,{\frac {{k}^{2}q \left( 8\,{t}^{2}{q}^{4}-4\,{q}^{4}+2\,{t}^{2}{q}^{2}-4\,{t}^{2}+1 \right) }{\sqrt {1-{q}^{2}} \left( {t}^{2}-1 \right) }} \cdot\mathbf{e}^{23}
+\frac{1}{2}\,{\frac {{k}^{2} \left( -3+2\,{q}^{2}+4\,{q}^{4} \right) }{{t}^{2}-1}} \cdot\mathbf{e}^{34}.
\end{split}
\]
Imposing that $-\riccib=\lambda\omega$, we find that $q=t=0$, so $d\omega=0$. Thus $g$ is static if and only if it is K\"ahler-Einstein.

This concludes the proof in the solvable case.
\vspace{3mm}

The non-solvable 4-dimensional Lie algebras $\mathfrak{su}(2)\times\R$ and $\mathfrak{sl} (2,\R) \times\R$ have structure equations
\[
\begin{split}
&\mathfrak{su}(2)\times\R\ = (- \mathbf{e}^{23},\mathbf{e}^{13},- \mathbf{e}^{12},0) \\
&\mathfrak{sl} (2,\R) \times\R\  = (- \mathbf{e}^{23} ,\mathbf{e}^{13},\mathbf{e}^{12},0).
\end{split}
\]
From a more general result in \cite{ras} we have that the only complex structures on these algebras are defined in both cases by
\[
J \mathbf{e}^1 =  \mathbf{e}^2, \qquad J \mathbf{e}^3 = -p\cdot  \mathbf{e}^3+(1+p^2)\cdot  \mathbf{e}^4.
\]
All the metrics compatible with those complex structures have the form 
\begin{equation}\label{metrCond}
\omega = m_{11}\mathbf{e}^{12}+m_{32}\mathbf{e}^{13}+\frac{m_{31}+pm_{32}}{1+p^2}\,\mathbf{e}^{14}-m_{31}\mathbf{e}^{23}+\frac{m_{32}-pm_{31}}{1+p^2}\,\mathbf{e}^{24}+m_{44}\mathbf{e}^{34}.
\end{equation}
Both these Lie algebras are semisimple, thus by Theorem 8 of \cite{chu} they cannot admit any invariant symplectic structure. So no invariant static metric with $\lambda\neq 0$ can be found on these algebras. Moreover, since they are unimodular, every 3-form is closed, then every $J$-compatible metric is SKT.

We study the two cases separately:
\begin{itemize}
\item $\mathfrak{su}(2)\times\R$.
Let $g$ be a Hermitian metric whose fundamental 2-form satisfy \eqref{metrCond}. Then
\[
\begin{split}
\riccib(X,Y) =&\, {\frac {-2\,{{ m_{31}}}^{2}-2\,{{m_{32}}}^{2}+(1+p^2)\,{ m_{44}}\,{ m_{11}}-{{ m_{44}}}^{2}(1+p^2)^2}{-{{ m_{31}}}^{2}-{{ m_{32}}}^{2}+{ m_{44}}\,{
 m_{11}}\,{p}^{2}+{ m_{44}}\,{ m_{11}}}} \cdot  \mathbf{e}^{12} \\
           &\, +\frac{1}{2}\,{\frac {(1+p^2)(m_{44}m_{32}+m_{11}m_{32}+p\cdot m_{44}m_{31} )}{-{{  m_{31}}}^{2}-{{  m_{32}}}^{2}+{  m_{44}}\,{
  m_{11}}\,{p}^{2}+{  m_{44}}\,{  m_{11}}}} \cdot ( \mathbf{e}^{13}+ \mathbf{e}^{24}) \\
           &\, +\frac{1}{2}\,{\frac {{  m_{44}}\,{p}^{2}{  m_{31}}+{  m_{44}}\,{  m_{31}}+p\,{  m_{11}
}{  m_{32}}+{  m_{31}}\,{  m_{11}}}{-{{  m_{31}}}^{2}-{{  m_{32}}}^{2}+{
  m_{44}}\,{  m_{11}}\,{p}^{2}+{  m_{44}}\,{  m_{11}}}} \cdot ( \mathbf{e}^{14}- \mathbf{e}^{23}). 
\end{split}
\]
As said before, this Lie algebra can only admit static metric with $\lambda=0$; imposing the vanishing of $\riccib$ we obtain that $m_{11}=(1+p^2)m_{44}$  and $m_{31}=m_{32}=0$, 
so every metric in the form
\begin{equation}\label{su(2)xR}
\omega= m_{44}(1+p^2)\mathbf{e}^{12}+m_{44}\mathbf{e}^{34}
\end{equation}
is static with $\lambda=0$. Therefore \emph{every complex structure on $\mathfrak{su}(2)\times\R$ admits a compatible static metric with $\lambda=0$}. More in general, for these metrics we have $\rho^B=0$.

\item $\mathfrak{sl} (2,\R) \times\R$.
Let $g$ be a Hermitian metric whose fundamental 2-form satisfy \eqref{metrCond}. Then
\[
\begin{split}
\riccib(X,Y) =&\, -{\frac {-2\,{{ m_{31}}}^{2}-2\,{{m_{32}}}^{2}+(1+p^2)\,{ m_{44}}\,{ m_{11}}+{{ m_{44}}}^{2}(1+p^2)^2}{-{{  m_{31}}}^{2}-{{  m_{32}}}^{2}+{  m_{44}}\,{
  m_{11}}\,{p}^{2}+{  m_{44}}\,{  m_{11}}}} \cdot  \mathbf{e}^{12} \\
           &\, +\frac{1}{2}\,{\frac { (1+p^2)(-m_{44}m_{32}+m_{11}m_{32}-p\cdot m_{44}m_{31} ) }{-{{  m_{31}}}^{2}-{{  m_{32}}}^{2}+{  m_{44}}
\,{  m_{11}}\,{p}^{2}+{  m_{44}}\,{  m_{11}}}} \cdot ( \mathbf{e}^{13}+ \mathbf{e}^{24}) \\
           &\, +\frac{1}{2}\,{\frac {-{  m_{31}}\,{  m_{44}}\,{p}^{2}-{  m_{44}}\,{  m_{31}}+p\,{  
m_{11}}{  m_{32}}+{  m_{31}}\,{  m_{11}}}{-{{  m_{31}}}^{2}-{{  m_{32}}}^{2
}+{  m_{44}}\,{  m_{11}}\,{p}^{2}+{  m_{44}}\,{  m_{11}}}} \cdot ( \mathbf{e}^{14}- \mathbf{e}^{23}). 
\end{split}
\]
Again, this algebra can only admit static metric with $\lambda=0$, and imposing that $\riccib=0$ we obtain that $m_{11}=-(1+p^2)m_{44}$ and $m_{31}=m_{32}=0$;
but $m_{11} m_{44}\leqslant 0$, that is a contradiction because $g$ is positive definite. Then $\mathfrak{sl} (2,\R) \times\R$ does not admit any static metric.
\end{itemize}
\end{proof}

\bigskip

\section{Tangent Lie algebras and compact semisimple Lie groups}\

Consider a $2n$-dimensional Lie algebra $\g$ and a flat connection $D$ of $\g$. We define the tangent Lie algebra $(T_D\,\g=\g\ltimes_D \R^{2n},[\,,]_D)$ with the Lie bracket 
\[
[(X_1,X_2),(Y_1,Y_2)]_D = ([X_1,Y_1],D_{X_1}Y_2-D_{Y_1}X_2).
\]
Additionally, if $(J,g)$ is a Hermitian structure on $\g$ and if $D$ is Hermitian, i.e. $Dg=DJ=0$, then on $T_D\,\g$ we can define a complex structure $\tilde J(X_1,X_2)=(JX_1,JX_2)$ and a $\tilde J$-Hermitian metric $\tilde g$ such that $(\g,0)$ and $(0,\g)$ are orthogonal (see \cite{bafi}). 
\begin{theorem}\label{tangent}
Let $(\g,J,g)$ be a Hermitian Lie algebra and $D$ a Hermitian flat connection. Then $g$ is a $J$-Hermitian static metric with $\lambda=0$ if and only if $\tilde g$ is a $\tilde J$-Hermitian static metric with $\lambda=0$.
\end{theorem}
\begin{proof}
By \cite[Proposition 3.1]{ef}, $(T_D\,\g,\tilde J,\tilde g)$ is SKT if and only if $(\g,J,g)$ is SKT, so we only have to prove that 
\[
\riccib = 0\ \Longleftrightarrow\ (\tilde\rho^B)^{(1,1)} = 0,
\] 
where $\tilde\rho^B$ is the Ricci tensor of the Bismut connection $\tilde \nabla^B$ of $(T_D\,\g,\tilde J,\tilde g)$. The Bismut connections of $(T_D\,\g,\tilde J,\tilde g)$ and $(\g,J,g)$ are related by
\begin{equation} \label{bismuttang}
\tilde g  (\tilde \nabla^B_{(X_1, X_2)} (Y_1, Y_2), (Z_1, Z_2)) =  g(\nabla^B_{X_1} Y_1, Z_1) + g (D_{X_1} Y_2, Z_2),
\end{equation}
so the curvature $\tilde R^B$ of $\tilde\nabla^B$ is given by $\tilde R^B=(R^B,R^D)=(R^B,0)$ since $D$ is flat. Hence, the Ricci tensors of $\tilde\nabla^B$ and $\nabla^B$ are equal, as well as their $(1,1)$ parts.

\end{proof}

Let $(M,J,g)$ be a complex manifold with an SKT metric such that the Bismut connection $\bismut$ has trivial holonomy. Then clearly $-\riccib = 0$, i.e. $g$ is a static metric with $\lambda=0$. It is well known that this condition holds if $M$ is a Lie group and $g$ a bi-invariant metric, that is a metric which is both left-invariant and right-invariant.
Let $\tilde g$ be the induced bi-invariant metric on the Lie algebra $\g$, then it satisfies 
\begin{equation}\label{eqbiinv}
\tilde g([X,Y],Z)=-\tilde g(Y,[X,Z]).
\end{equation}
By using this equation and the integrability of the complex structure in \eqref{eqbismut} and \eqref{eqtorsion} we obtain that $\nabla^B_XY=0$ for every $X,Y\in\g$, so $Hol(\nabla^B)=0$, and $c(X,Y,Z)=-\frac{1}{2}g([X,Y],Z)$. Then, applying \eqref{eqbiinv} and the Jacobi identity we obtain $dc=0$, so $\tilde g$ is SKT.

Since the work of Samelson and Wang \cite{sam,wang}, it has been known that every compact even-dimensional Lie group $G$ admits a left-invariant complex structure $J_L$ and a right-invariant one $J_R$. Moreover, if $G$ is semisimple, the bi-invariant metric $g_K$ induced by the Killing form is compatible with both $J_L,J_R$.
\begin{prop}\label{cpt}
Let $G$ be a compact, even-dimensional semisimple Lie group. Then it admits a static metric with $\lambda=0$.
\end{prop}
\begin{remark}
The Lie algebra $\mathfrak{su}(2)\times\R$ considered in the former section is contained in this class, and we can obtain a bi-invariant metric by setting $m_{44}=1$ and $p=0$ in \eqref{su(2)xR}.
\end{remark}

In view of this proposition, if we find a flat Hermitian connection $D$ on a $2n$-dimensional compact semisimple Lie group $G$ whose Lie algebra is $\g$, then we can construct a static metric with $\lambda=0$ on the tangent Lie algebra $T_D\,\g$.

\begin{example}
For every $2n$-dimensional compact semisimple Lie group $G\cong G_0\times S^1$, where $G_0$ is a $(2n-1)$-dimensional compact semisimple Lie group, we can construct a flat Hermitian connection $D$ (\cite[Proposition 3.4]{ef}).

If we consider the Lie algebra $\mathfrak{su}(2)\times\R$ together with the complex structure $J\mathbf{e}^1=\mathbf{e}^2,\ J\mathbf{e}^3=\mathbf{e}^4$ and the $J$-Hermitian bi-invariant metric $g=\sum_i \mathbf{e}^i\otimes \mathbf{e}^i$, we can define the Hermitian connection
\[
D_{\mathbf{e}_i}Y=0\ i=1,2,3 \qquad \ D_{\mathbf{e}_4}Y=JY
\]
for any $Y\in\g$. The corresponding tangent Lie algebra has structure equation
\[
T_D\,\g = \big( -\mathbf{f}^{23},\mathbf{f}^{13},-\mathbf{f}^{12},0,-\mathbf{f}^{46},\mathbf{f}^{45},-\mathbf{f}^{48},\mathbf{f}^{47} \big),
\]
where $\mathbf{f}^i=(\mathbf{e}^i,0)$ for $i=1,\dots,4$ and $\mathbf{f}^j=(0,\mathbf{e}^j)$ for $j=5,\dots,8$, and the induced Hermitian structure $(\tilde J,\tilde g)$ is given by
\[
\tilde J\mathbf{f}^{2i-1}=\mathbf{f}^{2i},\ i=1..4 \qquad \tilde g=\sum_{i=1}^8 \mathbf{f}^i\otimes \mathbf{f}^i.
\]
\end{example}

Note that for any other compact semisimple Lie group this result cannot be applied since for every non-abelian simple Lie algebra $[\g,\g]=\g$. However, this is not the only way to construct flat Hermitian connections, as shown in the next example.

\begin{example}
Let us consider $G=S^3\times S^3$. The associated Lie algebra $\g$ is 
\[
\mathfrak{su}(2)\times \mathfrak{su}(2) = \big( -\mathbf{e}^{23},\mathbf{e}^{13},-\mathbf{e}^{12},-\mathbf{e}^{56},\mathbf{e}^{46},-\mathbf{e}^{45} \big).
\]
Clearly $[\g,\g]=\g$, so we cannot apply Proposition 3.4 of \cite{ef}. However, the linear connection $D$ defined by
\[
\begin{split}
 &D_{\mathbf{e}_1}=\left[
 \begin {array}{cccccc} 
 0&0&0&\frac12&0&\frac{1}{\sqrt{2}}\\
 0&0&-\frac12&0&-\frac{1}{\sqrt{2}}&0\\
 0&\frac12&0&0&-\frac12&0\\
 -\frac12&0&0&0&0&-\frac12\\
 0&\frac{1}{\sqrt{2}}&\frac12&0&0&0\\
 -\frac{1}{\sqrt{2}}&0&0&\frac12&0&0
 \end {array}
 \right] \qquad\quad\,
D_{\mathbf{e}_2} = \left[ 
 \begin {array}{cccccc} 
 0&-\frac12&0&0&-\frac12&0\\ 
 \frac12&0&0&0&0&-\frac12\\
 0&0&0&1&0&0\\
 0&0&-1&0&0&0\\ 
 \frac12&0&0&0&0&-\frac12\\
 0&\frac12&0&0&\frac12&0
 \end {array} 
 \right] \\
&D_{\mathbf{e}_3} = \left[
 \begin {array}{cccccc} 
 0&-\frac{1}{\sqrt{2}}&\frac12&0&0&0\\
 \frac{1}{\sqrt{2}}&0&0&\frac12&0&0\\
 -\frac12&0&0&0&0&-\frac12\\
 0&-\frac12&0&0&\frac12&0\\
 0&0&0&-\frac12&0&\frac{1}{\sqrt{2}}\\
 0&0&\frac12&0&-\frac{1}{\sqrt{2}}&0
 \end {array}
 \right] \quad
D_{\mathbf{e}_4}=D_{\mathbf{e}_5}=D_{\mathbf{e}_6}= \left[
 \begin {array}{cccccc} 
 0&0&0&0&0&0\\
 0&0&0&0&0&0\\
 0&0&0&0&0&0\\
 0&0&0&0&0&0\\
 0&0&0&0&0&0\\
 0&0&0&0&0&0
 \end {array}
 \right]
\end{split}
\]
is flat and Hermitian, so we can apply Theorem \ref{tangent} and construct the tangent Lie algebra $T_D\,\g$ over $\mathfrak{su}(2)\times \mathfrak{su}(2)$.

\end{example}

\end{document}